\documentclass[12pt,a4paper]{amsart}
\usepackage{amsmath,amsthm,amssymb,amsfonts}
\usepackage{tikz-cd}
\usepackage{thmtools,xcolor}
\usepackage[bookmarks=true,hyperindex,pdftex,colorlinks,citecolor=blue,
linkcolor=blue,urlcolor=blue]{hyperref}
\usepackage{stackrel}
\usepackage{xypic}

\usepackage[english]{babel}
\usepackage{enumitem}
\usepackage{bm}
\usepackage{accents}
\usepackage{graphicx}

\usepackage{setspace}
\declaretheoremstyle[
headfont=\color{blue}\normalfont\bfseries,
bodyfont=\color{blue}\normalfont\itshape,
]{colored}

\theoremstyle{plain}
\newtheorem{theorem}{Theorem}[section]

\newtheorem{prop}[theorem]{Proposition}
\newtheorem{lemma}[theorem]{Lemma}

\theoremstyle{definition}

\newtheorem{remark}[theorem]{Remark}
\newtheorem{definition}[theorem]{Definition}

\newcommand{\vertiii}[1]{{\left\vert\kern-0.25ex\left\vert\kern-0.25ex\left\vert #1
 \right\vert\kern-0.25ex\right\vert\kern-0.25ex\right\vert}}

\newcommand{\R}{\mathbb{R}}
\newcommand{\N}{\mathbb{N}}
\newcommand{\C}{\mathbb{C}}

\newcommand{\D}{\mathbb{D}}
\newcommand{\A}{\mathcal{A}}
\newcommand{\M}{\mathcal{M}}

\newcommand{\eps}{\varepsilon}
\newcommand{\Ho}{\mathcal{H}}

\newcommand{\Au}{\mathcal{A}_u }

\renewcommand{\leq}{\leqslant}
\renewcommand{\geq}{\geqslant}

\title[Algebras of holomorphic functions on polydisk type domains]{The spectra of Banach algebras of holomorphic functions on polydisk type domains}

\author[Y. S. Choi]{Yun Sung Choi}
\address[Yun Sung Choi]{Department of Mathematics, POSTECH, Pohang 790-784, Republic of Korea}
\email{\texttt{mathchoi@postech.ac.kr}}

\author[M. Jung]{Mingu Jung}
\address[Mingu Jung]{Department of Mathematics, POSTECH, Pohang 790-784, Republic of Korea \newline
	\href{http://orcid.org/0000-0003-2240-2855}{ORCID: \texttt{0000-0003-2240-2855} }}
\email{\texttt{jmingoo@postech.ac.kr}}

\author[M. Maestre]{Manuel Maestre}
\address[Manuel Maestre]{Departamento de An\'{a}lisis Matem\'{a}tico,
	Universidad de Valencia, Doctor Moliner 50, 46100 Burjasot
	(Valencia), Spain}
\email{manuel.maestre@uv.es}

\thanks{The first and second author were supported by the Basic Science Research Program through the National Research Foundation of Korea (NRF) funded by the Ministry of Education (NRF-2019R1A2C1003857). The third author was supported by MINECO and FEDER Project MTM2017-83262-C2-1-P, and also supported by Prometeo PROMETEO/2017/102.}

\date{\today}


\begin{document}
	
	\begin{abstract}
R.M. Aron et al. \cite{ACGLM} proved that the Cluster Value Theorem in the infinite dimensional Banach space setting holds for the Banach algebra $\Ho^\infty (B_{c_0})$. On the other hand, B.J. Cole and T.W. Gamelin \cite{CG} showed that $\Ho^\infty (\ell_2 \cap B_{c_0})$ is isometrically isomorphic to $\Ho^\infty (B_{c_0})$ in the sense of an algebra. Motivated by this work, we are interested in a class of open subsets $U$ of a Banach space $X$ for which $\Ho^\infty (U)$ is isometrically isomorphic to $\Ho^\infty (B_{c_0})$.
We prove that there exist polydisk type domains $U$ of any infinite dimensional Banach space $X$ with a Schauder basis such that $\Ho^\infty (U)$ is isometrically isomorphic to $\Ho^\infty (B_{c_0})$, which generalizes the result by Cole and Gamelin \cite{CG}. Furthermore, we study the analytic and algebraic structure of the spectrum of $\Ho^\infty (U)$ and show that the Cluster Value Theorem is true for $\Ho^\infty (U)$.
	\end{abstract}
	
	\maketitle
	
	\section{Introduction}

In 1961, I.J. Schark \cite{S} proved the following result.
\begin{theorem}
For each $f \in \Ho^\infty (\D)$ and $z \in \overline{\D}$, the image of the fiber $\M_z (\Ho^\infty (\D ))$ under the Gelfand transform $\widehat{f}$ of $f$ coincides with the set of all limits of $(f(x_{\alpha}))$ for the nets $(x_{\alpha})\subset D$ converging to $z$.
\end{theorem}
All the definitions and unexplained notations can be found in Section \ref{Background}. This result has been called the Cluster Value Theorem. It is also known as the weak version of the Corona Theorem, which was proved by L. Carleson \cite{C}. In the infinite dimensional Banach space setting, R.M. Aron et al. \cite{ACGLM} presented in 2012 the first positive results for the Cluster Value Theorem by proving that the theorem is true for the Banach algebra $\Ho^\infty (B_{c_0})$.
Though they showed that the Cluster Value Theorem is also true for the closed subalgebra $\mathcal{A}_u (B_{\ell_2})$ of $\Ho^\infty (B_{\ell_2})$, it is still open if this theorem holds for $\Ho^\infty (B_{\ell_2})$. Surprisingly, a domain $U$ of a Banach space has not been found for which the Cluster Value Theorem does not hold.

Three years later, W.B. Johnson and S.O. Castillo \cite{JC} proved that the Cluster Value Theorem holds also for $C(K)$-spaces, when $K$ is a scattered compact Hausdorff space by using similar arguments as in \cite[Theorem 5.1]{ACGLM}. However, up to our knowledge, $c_0$ and $C(K)$ with $K$ a compact Hausdorff scattered, are the only known Banach spaces $X$ for which the Cluster Value Theorem holds for $\Ho^\infty (B_X)$. For historical background and recent developments on this topic, we refer to the survey article \cite{CGMP}.

On the other hand, we got interested in more general open domains $U$ of a Banach space $X$ than its open unit ball $B_X$ such that the Cluster Value Theorem holds for $\Ho^\infty (U)$. It was observed by B.J. Cole and T.W. Gamelin \cite{CG} that the Banach algebra $\Ho^\infty (\ell_2 \cap B_{c_0})$ is isometrically isomorphic to $\Ho^\infty (B_{c_0})$. Motivated by this work, we first generalize the open domain $\ell_2\cap B_{c_0}$ to a polydisk type domains $U$ of a Banach space $X$ with a Schauder basis for which the Banach algebra $\Ho^\infty (U)$ is isometrically isomorphic to the Banach algebra $\Ho^\infty (B_{c_0})$. Once this is established, we can derive the analytic and algebraic structure of the spectrum $\M (\Ho^\infty (U))$ from those of $\M (\Ho^\infty (B_{c_0}))$ with the aid of the transpose of the aforementioned isometric isomorphism.
For instance, by applying the Cluster Value Theorem for $\Ho^\infty (B_{c_0})$ (see \cite[Theorem 5.1]{ACGLM}), we would find a new open domain $U$ for which the Cluster Value Theorem holds for $\Ho^\infty (U)$. We will introduce some polydisk type domains of a Banach space $X$ with a Schauder basis, which will be denoted by $\D_X (\bm{r})$, where $\bm{r}$ is a sequence of positive real numbers (see Definition \ref{def:rD_X}). As a matter of fact, the domain $\ell_2 \cap B_{c_0}$ considered in \cite{CG} turns to be a special case of $\D_X (\bm{r})$.

Let us describe the contents of the paper: Section \ref{Background} is devoted to basic materials and the definition of certain polydisk type domains $\D_X (\bm{r})$, which are of our interest.
In Section \ref{sec:isometric_isomorphism}, we prove that the Banach algebra $\Ho^\infty (\D_X (\bm{r}))$ of all bounded holomorphic functions on $\D_X (\bm{r})$ (endowed with the supremum norm on $\D_X (\bm{r})$) is isometrically isomorphic to $\Ho^\infty (B_{c_0})$ as algebras under some assumption on $\bm{r}$, which generalizes the result in \cite{CG}.
As all the algebras $\Ho^\infty (\D_X (\bm{r}))$ are isometrically isomorphic, without loss of generality, we may fix the sequence $\bm{r}$ to be $(1,1,\ldots)$ and denote the corresponding domain by $\D_X$ in the following sections.
As we commented above, in Section \ref{sec:fiber}, we study the structure of the spectrum $\M (\Ho^\infty (\D_X))$ by applying that of the spectrum $\M (\Ho^\infty (B_{c_0}))$ through the transpose of the above isometric isomorphism. Moreover, we consider the notion of a cluster set of elements in $ \Ho^\infty (\D_X)$, and examine its relation with that in $\Ho^\infty (B_{c_0})$. Using this relation, we prove that the corresponding Cluster Value Theorem holds for $\Ho^\infty (\D_X)$.

\section{Background}\label{Background}

In this section, we provide some necessary notations. Throughout this paper, $X$ will be an infinite dimensional complex Banach space with open unit ball $B_X$. The algebras that we are going to consider are denoted by $\Ho^\infty (U)$ and $\mathcal{A}_u (U)$, the algebra of all bounded holomorphic functions on an open subset $U$ of $X$ and the algebra of all uniformly continuous holomorphic functions on $U$, respectively. Note that $\Ho^\infty (U)$ becomes a Banach algebra when it is endowed with the supremum norm on $U$ and so does $\mathcal{A}_u (U)$ whenever $U$ is convex and bounded. For a Banach algebra $\mathcal{A}$, we will denote by $\mathcal{M} (\mathcal{A})$ the spectrum (maximal ideal space) of $\mathcal{A}$, that is, the space of all continuous homomorphisms on $\mathcal{A}$. Notice that the spectrum $\M (\A)$ is always a compact set endowed with the restriction of the weak-star topology. For each $a \in \mathcal{A}$, the Gelfand transform of $a$ will be denoted by $\widehat{a}$ that sends $\phi \in \mathcal{M} (\mathcal{A})$ to $\phi (a) \in \C$ for every $\phi \in \mathcal{M} (\mathcal{A})$.

If a Banach algebra $\mathcal{A}$ is taken to be $\Ho^\infty (B_X)$ or $\A_u (B_X)$, then there is the natural restriction mapping $\pi$ on $\M (\A )$ to $X^{**}$ given as $\pi (\phi) = \phi \vert_{X^*}$ because the dual space $X^*$ is included in $\A$ in this case. By Goldstine's theorem, we have that $\pi (\M (\A )) = \overline{B}_{X^{**}}$ and the \emph{fiber of $\M (\A)$ at $z \in \overline{B}_{X^{**}}$} can be well defined as the pre-image of $z$ under the mapping $\pi$, i.e.,
\[
\M_z (\A) = \{ \phi \in \A : \pi (\phi) = z \} = \pi^{-1} (z).
\]
We say that the \emph{Cluster Value Theorem holds for $\mathcal{A}$ at $z \in \overline{B}_{X^{**}}$} when the image of the fiber $\M_z (\mathcal{A})$ under the Gelfand transform $\widehat{f}$ coincides with the \emph{cluster sets $Cl (f, z)$} defined by
\[
\{ \lambda \in \C: \, \text{there exists a net} \,\, (x_{\alpha}) \subset B_X \,\, \text{such that} \,\, x_{\alpha} \xrightarrow{w^*} z \,\, \text{and} \,\, f(x_\alpha) \rightarrow \lambda \},
\]
for every $f \in \mathcal{A}$. If $\widehat{f} (\M_z (\mathcal{A}))$ coincides with $Cl (f,z)$ for every $z \in \overline{B}_{X^{**}}$ and every $f \in \mathcal{A}$, then we simply say that the \emph{Cluster Value Theorem holds for $\mathcal{A}$}.
It is proved in \cite{ACGLM} that the Cluster Value Theorem holds for $\Ho^\infty (B_{c_0})$ and for $\mathcal{A}_u (B_{\ell_2})$, and that the Cluster Value Theorem holds for $\mathcal{A}_u (B_X)$ at $0$ whenever $X$ is a Banach space with a shrinking $1$-unconditional basis.

However, unlike the situation when we dealt with $\Ho^\infty (B_X)$ or with $\A_u (B_X)$, the dual space $X^*$ is not included in $\Ho^\infty (U)$ in general for an open subset $U$ of a Banach space (see Proposition \ref{Au(D_X)_not_contained_in_Ho(D_X)}).
Thus, `fibers' of $\M (\Ho^\infty (U))$ cannot be defined in the same manner as before. For this reason, we consider the subspace $X^\sharp$ of $\Ho^\infty (U)$ defined as
\begin{equation}\label{X_sharp_U}
X^\sharp = \left(\Ho^\infty (U) \cap X^*, \| \cdot \|_U \right),
\end{equation}
where $\| \cdot \|_U$ is the supremum norm on $U$. In other words, $X^\sharp$ is the space which consists of elements of $\Ho^\infty (U)$ which are bounded linear functionals on $X$.
Notice that $X^\sharp$ can possibly be the trivial set $\{ 0 \}$. Under the case when $X^\sharp \neq \{ 0 \}$, we can consider the restriction mapping $\pi$ from the spectrum $\M(\Ho^\infty (U))$ to the closed unit ball $\overline{B}_{(X^\sharp)^*}$ of the dual space $(X^\sharp)^*$. Now, we consider the following definition.

\begin{definition}\label{def:fibers_U}
Let $X$ be a Banach space and $U$ be an open subset of $X$. If the restriction mapping $\pi : \M(\Ho^\infty (U)) \rightarrow \overline{B}_{(X^\sharp)^*}$ is surjective, then the \emph{fiber of $\M (\Ho^\infty (U))$ at $z \in \overline{B}_{(X^\sharp)^*}$} is defined as $\{ \phi \in \M (\Ho^\infty (U)) : \pi (\phi) = z \}$.
\end{definition}

As the above definition might be too general for our purpose, we would like to focus on certain polydisk type domains instead of dealing with an arbitrary open subset of a Banach space. Let us note that the following domain can be seen as a generalization of the domains $B_{c_0}$ (when $X = c_0$) or $\ell_p \cap \D^\N$, $1\leq p < \infty$ (when $X = \ell_p$).

\begin{definition}\label{def:rD_X}
Suppose that $X$ is a Banach space with a normalized Schauder basis and let $\bm{r} = (r_n)_{n =1}^\infty$ be a sequence of positive real numbers. Let us consider the set $\D_X (\bm{r})$ given by
\begin{equation*}
\D_X (\bm{r})= \left\{ x \in X : x = \sum_{j=1}^{\infty} e_j^* (x) e_j \,\, \text{with} \,\, |e_j^*(x)| < r_j \,\, \text{for every} \,\, j \in \N \right\}.
\end{equation*}
In particular, when $\bm{r} = (1, 1, \ldots)$, we simply denote it by $\D_X$.
\end{definition}

Let us make some remarks on $\D_X (\bm{r})$. If $\inf\limits_{n\in\N} r_n = 0$, then $\D_X (\bm{r})$ is never open in $X$. For if it were, it would follow that $t B_X \subset \D_X (\bm{r})$ for some $t > 0$. For $N \in \N$ satisfying that $r_N < \frac{t}{2}$, we have that $\frac{t}{2} e_N \in t B_X \subset \D_X (\bm{r})$. This implies that $\frac{t}{2} < r_N < \frac{t}{2}$, a contradiction.
As a matter of fact, we have the following result. Recall that for a Banach space with a normalized Schauder basis $(e_j)_{j=1}^{\infty}$ and coefficient functionals $(e_j)_{j=1}^{\infty}$, we always have $\sup_{j \in \N} \|e_j^* \| < \infty$.

\begin{remark}
The set $\D_X (\bm{r})$ is an open subset of $X$ if and only if $\inf\limits_{n\in \N} r_n > 0$.
\end{remark}

\begin{proof}
We only need to prove if $\inf\limits_{n\in\N} r_n > 0$ then $\D_X (\bm{r})$ is open. To this end, put $\alpha = \inf\limits_{n\in\N} r_n > 0$ and fix $x \in \D_X (\bm{r})$. Let $N \in \N$ be such that $|e_j^* (x)| < \frac{\alpha}{2}$ for all $j \geq N$. Choose $\eps > 0$ sufficiently small so that $\eps < \frac{\alpha}{2}$ and $\eps < \min\{ r_j - |e_j^* (x)| : 1 \leq j \leq N-1\}$. If $y \in X$ satisfies $\| x - y \| < (\sup\limits_{j \in \N} \|e_j^*\| )^{-1} \eps$, then $|e_j^* (y)| < r_j$ for every $j \in \N$.
\end{proof}

\begin{remark}
	M. I. Kadec in 1967 proved that given $X$ an infinite dimensional real or complex separable Banach space, $X$ and its open unit ball $B_X$ are isomorphic to $\mathbb{R}^{\mathbb{N}}$ the topological product of a countable number of copies of $\mathbb{R}$  (see e.g. \cite[Theorem 5.1, p. 188]{BP}). An obvious consequence is that if $X$ and $Y$ are infinite dimensional real or complex separable Banach space, then $B_X$ and $B_Y$ are homeomorphic. In this sense, the following natural question is asked by the referee: Given $X$ and $Y$ two \emph{infinite dimensional} Banach spaces with Schauder bases, and given  $\bm{r} = (r_n)_{n =1}^\infty$ and $\bm{s} = (s_n)_{n =1}^\infty$ two sequences of positive real numbers such that $\inf\limits_{n\in\N} r_n > 0$ and $\inf\limits_{n\in\N} s_n > 0$, are $\D_X (\bm{r})$ and $\D_Y (\bm {s})$ homeomorphic? The answer is positive and to give it, we use the following lemma.
\end{remark}

\begin{lemma}\label{lem:homeo}
Let $X$ be a Banach space and $U$ a non empty convex balanced open subset of $X$ which does not contain any non-trivial subspace (i.e., for any $x \neq 0$, there exists $\lambda_x \neq 0$ such that $\lambda_x x \notin U$). Then $U$ is homeomorphic to $X$.
\end{lemma}

\begin{proof}
Let us denote by $p : X \rightarrow \mathbb{R}$, defined by $p(x)=\inf\{\lambda >0: x\in \lambda U\}$. We have that $U=\{x\in X: p(x)<1\}$. By assumption, we have that if $x \neq 0$, then $p(x)>0$. In particular, $p$ is a norm and $|p(x)-p(y)| \leq p(x-y)$ for all $x, y \in X$. Without loss of generality, we may assume that $B_X \subseteq U$. Thus, $|p(x)-p(y)| \leq \|x-y\|$ for all $x, y \in X$ and $p$ is a $\|\cdot\|$-continuous function. As a consequence  the mapping  $h : U \rightarrow X$ given by $h(x) = (1-p(x))^{-1} x$ for every $x \in U$ is $\|\cdot\|$-continuous and since $h^{-1} : X \rightarrow U$ is given by $h^{-1} (z) = (1+ p(z))^{-1}{ z},$ $h^{-1}$ is also continuous, we complete the proof.
\end{proof}

\begin{prop}
Let $X$ and $Y$ be {infinite dimensional} Banach spaces with Schauder bases, and $\bm{r} = (r_n)_{n =1}^\infty$ and $\bm{s} = (s_n)_{n =1}^\infty$ be sequences of positive real numbers such that $\inf\limits_{n\in\N} r_n > 0$ and $\inf\limits_{n\in\N} s_n > 0$. Then $\D_X (\bm{r})$ and $\D_Y (\bm {s})$ are homeomorphic.
\end{prop}

\begin{proof}
Observe from Lemma \ref{lem:homeo} that $X$ and $Y$ are homeomorphic to $\D_X (\bm{r})$ and $\D_Y (\bm{s})$, respectively, and that by the above mentioned Kadec's result, $X$ and $Y$ are homeomorphic.
\end{proof}

Recall that a \emph{complete Reinhardt domain} $R$ of a Banach space $X$ with a normalized unconditional Schauder basis is a subset such that for every $x \in R$ and $y \in X$ with $|e_j^* (y)| \leq |e_j^* (x)|$ for every $j \in \N$, we have $y \in R$.
Notice that if a basis of a Banach space $X$ is unconditional, then $\D_X (\bm{r})$ becomes a complete Reinhardt domain.
For more theory about holomorphic functions on Reinhardt domains of a Banach space, see \cite{CMV, DeGaMaSe, MN} and the references therein. Nevertheless, most of results in this paper hold for a general normalized Schauder basis.

\section{An isometric isomorphism between $\Ho^\infty (\D_X (\pmb{r}))$ and $\Ho^\infty (B_{c_0})$}\label{sec:isometric_isomorphism}

We start with the following main result which can produce many examples of (bounded or unbounded) polydisk type domains $U$ of a Banach space for which $\Ho^\infty (U)$ is isometrically isomorphic to $\Ho^\infty (B_{c_0})$.
Using the standard notation, let us denote by $P_n$ the canonical projection from a Banach space $X$ with Schauder basis $(e_j)_{j=1}^{\infty}$ onto the linear subspace spanned by $\{e_j : 1\leq j \leq n\}$ for each $n\in\mathbb{N}$.

\begin{theorem}\label{Hol:D(X):B_{c_0}} Let $X$ be an infinite dimensional Banach space with a normalized Schauder basis $(e_j)_{j=1}^{\infty}$, and $\bm{r} = (r_n)_{n=1}^\infty$ a sequence of positive real numbers with $\inf\limits_{n\in\N} r_n >0$. Then $\Ho^\infty (\D_X (\bm{r}))$ and $\Ho^\infty (B_{c_0})$ are isometrically isomorphic.
\end{theorem}

\begin{proof}
	Let us denote by $\iota$ the linear injective mapping from $X$ to $c_0$ defined as
	$\iota (x) = \left( r_j^{-1} e_j^* (x) \right)_{j=1}^{\infty}$ for $x = \sum_{j=1}^{\infty} e_j^* (x) e_j.$
	Then $\iota$ maps the set $\D_X (\bm{r})$ into $B_{c_0}$.
	Consider the mapping $\Psi : \Ho^\infty ( B_{c_0} ) \rightarrow \Ho^\infty ( \D_X (\bm{r}))$ defined as
	\[
	\Psi (f) (x) = (f \circ \iota) (x) \quad \text{for every} \,\, f \in \Ho^\infty (B_{c_0}) \,\, \text{and} \,\, x \in \D_X (\bm{r}).
	\]
	It is clear that $\Psi$ is well-defined.
		
	To check that $\Psi$ is injective, suppose $\Psi (f) = 0$ for some $f\in \Ho^\infty (B_{c_0})$. If $z=(z_j)_{j=1}^{N} \in B_{c_{00}}$, then it is obvious that $x'= \sum_{j=1}^{N} r_j z_j e_j \in \D_X (\bm{r})$ and $\iota (x') = z$. So, $f(z) = (f \circ \iota) (x') = \Psi (f) (x') =0$ which implies that $f = 0$ on $B_{c_{00}}$. Noting that $B_{c_{00}}$ is dense in $B_{c_0}$, we have that $f =0$ on $B_{c_0}$.
	
	We claim that $\Psi$ is surjective. Let us denote by $\kappa$ the natural inclusion mapping from $c_{00}$ into $X$ mapping $z = (z_j)_{j =1}^\infty \in {c_{00}}$ to $\sum_{j=1}^{\infty} r_j z_j e_j$ (a finite sum) and let $g \in \Ho^\infty (\D_X (\bm{r}))$ be given.
	Notice that the mapping $\kappa$ is not continuous in general, nevertheless, we claim that the function $g \circ \kappa$ is holomorphic on $B_{c_{00}}$. To this end, as $g \circ \kappa$ is bounded on $B_{c_{00}}$, it suffices to check that $g \circ \kappa$ is G\^ateaux holomorphic on $B_{c_{00}}$ \cite[Theorem 15.35]{DeGaMaSe}. Fix $a = (a_j)_{j=1}^\infty \in B_{c_{00}}$ and $w=(w_j)_{j=1}^\infty \in {c_{00}}$, and consider the mapping $\lambda \mapsto (g \circ \kappa )(a + \lambda w)$ on the set $ \Omega := \{ \lambda \in \C : z_0 + \lambda w_0 \in B_{c_{00}} \}$. Note that
	\[
	(g \circ \kappa) (a + \lambda w) = g \left( \sum_{j=1}^m r_j a_j e_j + \lambda \sum_{j=1}^m r_j w_j e_j \right)
	\]
	for some $m \in \N$, and it is holomorphic in the variable $\lambda \in \Omega$ since $g$ is holomorphic.
	Now, consider $g \circ \kappa \vert_{\D^N} : \D^N \rightarrow \C$ for each $N \in \N$. Note that the domain $\D^N$ is endowed with the supremum norm induced from $c_{00}$.
	 As the function $g \circ \kappa \vert_{\D^N}$ is separately holomorphic, it is holomorphic on $\D^N$ in classical several variables sense due to Hartogs' theorem.
	From this, we have a unique family $(c_{\alpha, N} (g))_{\alpha \in \N_0^N}$ such that
	\[
	\left( g \circ \kappa \vert_{\D^N} \right)(z) = \sum_{\alpha \in \N_0^N} c_{\alpha, N} (g) z^{\alpha} \quad (z \in \D^N).
	\]
	It follows that there exists a unique family $(c_{\alpha}(g))_{\alpha \in \N_0^{(\N)}}$ such that
	\[
	(g \circ \kappa)(z) = \sum_{\alpha \in \N_{0}^{(\N)}} c_{\alpha} (g) z^{\alpha} \quad (z \in B_{c_{00}}).
	\]
	Also, note that
	\[
	\sup_{N \in \N} \sup_{z \in \D^N} \left| \sum_{\alpha \in \N_0^N} c_{\alpha} (g) z^{\alpha} \right| \leq \sup_{N\in\N} \sup_{z \in \D^N} |\left( g \circ \kappa \vert_{\D^N} \right)(z)| \leq \|g\|_{\D_X (\bm{r})} < \infty.
	\]
	By Hilbert's criterion (see \cite[Theorem 2.21]{DeGaMaSe}), there exists a unique $f \in \Ho^\infty (B_{c_0})$ such that $c_{\alpha} (f) = c_{\alpha} (g)$ for every $\alpha \in \N_{0}^{(\N)}$ with
	\[
	\|f\|_{\infty} := \sup_{x \in B_{c_0}} |f(x)| = \sup_{N \in \N} \sup_{z \in \D^N} \left| \sum_{\alpha \in \N_0^N} c_{\alpha} (g) z^{\alpha} \right|.
	\]
	Hence $f =g \circ \kappa$ on $B_{c_{00}}$ which implies that
	\begin{align*}
	\Psi (f) (x) = f(\iota (x))
	&= \lim_{n \rightarrow \infty } (g \circ \kappa ) ( ( \iota \circ P_n) (x) ) \quad (\text{since } (\iota \circ P_n )(x) \in B_{c_{00}}) \\
	&= \lim_{n \rightarrow \infty } g(P_n (x)) = g(x)
	\end{align*}
	for every $x \in \D_X (\bm{r})$, i.e., $\Psi (f) = g$.
	As
	\[
	\|f\|_{\infty} \leq \|g\|_{\D_X (\bm{r})} = \| \Psi(f)\|_{\D_X (\bm{r})} \leq \|f\|_{\infty},
	\]
	where the last inequality is obvious by definition of $\Psi$, we see that $\Psi$ is an isometry. It is clear by definition that $\Psi$ is a multiplicative mapping.
\end{proof}


From this moment on, we shall deal only with the domain $\D_X$ since all results on the Banach algebra $\Ho^\infty (\D_X)$ can be translated to a Banach algebra $\Ho^\infty (\D_X (\bm{r}))$ with $\inf\limits_{n\in\N} r_n >0$ with the aid of the isometric isomorphism in Theorem \ref{Hol:D(X):B_{c_0}}.

Throughout this paper, we denote by $\Psi$ the isometric isomorphism from $\Ho^\infty (B_{c_0})$ onto $\Ho^\infty ( \D_X)$ given in the proof of Theorem \ref{Hol:D(X):B_{c_0}}. We also keep the notations $\iota : X \rightarrow c_0$ and $\kappa : c_{00} \rightarrow X$ defined in the above with respect to the sequence $\bm{r} = (1,1,\ldots)$. We also let ${\Phi}$ denote the restriction of the adjoint of $\Psi$ to $\M (\Ho^\infty (\D_X ))$. Note that the mapping ${\Phi}$ is an homeomorphism from $\M (\Ho^\infty ( \D_X ))$ onto $\M (\Ho^\infty (B_{c_0}))$.

From the fact that every continuous polynomial on $c_0$ is weakly continuous on bounded sets (see \cite[Proposition 1.59]{D} or \cite[Section 3.4]{G}), we first see that if a continuous polynomial on $X$ is bounded on $\D_X$, then it is weakly continuous on bounded sets.

\begin{prop}
Let $X$ be a Banach space with a normalized Schauder basis and let $P$ be a continuous polynomial on $X$. If $P$ is bounded on $\D_X$, then $P$ is weakly continuous on bounded sets.
\end{prop}

\begin{proof}
Without loss of generality, we may assume that $P$ is an $n$-homogenous polynomial. To prove that $P$ is weakly continuous on bounded sets, it suffices by homogeneity to prove that $P$ is weakly continuous on some ball $r B_X$. Let $(x_\alpha) \subset \big(\sup\limits_{j \in \N} \|e_j^* \| \big)^{-1} B_X$ be a net converging weakly to some $z$. As $P$ is bounded on $\D_X$, it can be observed that $Q:= \Psi^{-1} ( P \vert_{\D_X}) \in \Ho^\infty (B_{c_0})$ is an $n$-homogeneous polynomial on ${c_0}$. So, $Q$ is weakly continuous on bounded sets. Now,
\begin{align*}
P(x_\alpha) = (P \vert_{\D_X}) (x_\alpha) &= \Psi (Q) (x_\alpha) = Q ( \iota (x_\alpha) ) \rightarrow Q ( \iota(z)) = P (z)
\end{align*}
since $\iota (x_\alpha) \xrightarrow{w(c_0, \,\ell_1)} \iota(z)$.
\end{proof}

Recall that if $U$ is a convex bounded open set in $X$, then $\Au (U)$ becomes a subalgebra of $\Ho^\infty (U)$ as uniformly continuous functions on convex bounded sets of a Banach space are bounded.
Thus, $\Au (B_{c_0})$ is a subalgebra of $\Ho^\infty (B_{c_0})$ and observe that the isometric isomorphism $\Psi$ from $\Ho^\infty (B_{c_0})$ onto $\Ho^\infty ( \D_X)$ maps $\Au (B_{c_0})$ into the algebra $\Au (\D_X)$. 
In other words, $\Psi ( \Au (B_{c_0})) \subset \Au (\D_X)$.
So, one might ask if Theorem \ref{Hol:D(X):B_{c_0}} holds for the Banach algebras $\Au (B_{c_0})$ and $\Au (\D_X)$, i.e., whether $\Au (B_{c_0})$ is isometrically isomorphic to $\Au (\D_X)$. The following observations give a negative answer to this question. Moreover, we see that $\D_X$ is not bounded in general and that $\Au (\D_X)$ is not a subalgebra of $\Ho^\infty (\D_X)$.

We would like to thank Daniel Carando for pointing out to us the following result.

\begin{prop}[D. Carando]\label{prop:unbounded_D_X}
If $X$ is a Banach space with a normalized Schauder basis $(e_j)_{j=1}^{\infty}$, then the set $\D_X$ is bounded if and only if $X$ is isomorphic to $c_0$.
\end{prop}

\begin{proof}
It is clear that if $X$ is isomorphic to $c_0$, then $\D_X$ is bounded. Assume that $\D_X$ is bounded. Then there exists $R > 0$ such that $\D_X \subset R B_X$. This implies that if $x \in X$ satisfies $\sup_{j \in \mathbb{N}} |e_j^* (x)| < 1$ then $\|x\| < R$. It follows that
\[
(\sup_{j \in \mathbb{N}} \|e_j^*\|)^{-1} \sup_{j \in \N} |e_j^* (x)| \leq \| x \| \leq R \sup_{j \in \N} |e_j^* (x)|
\]
for every $x \in X$. Thus, $(e_j)_{j \in \N}$ is equivalent to the canonical basis of $c_0$.
\end{proof}

\begin{prop}\label{Au(D_X)_not_contained_in_Ho(D_X)}
If $X$ is a Banach space with a normalized Schauder basis, then
$\Psi ( \Au (B_{c_0})) \subset \Au (\D_X) \cap \Ho^\infty (\D_X)$. But the space $\Au (\D_X)$ is contained in $\Ho^\infty (\D_X)$ if and only if $X$ is isomorphic to $c_0$.
\end{prop}

\begin{proof}
The first assertion follows from Theorem \ref{Hol:D(X):B_{c_0}}. By the above proposition, if $X$ is not isomorphic to $c_0$, the set $\D_X$ is not bounded, so not weakly bounded. Thus, there exists $x^* \in X^*$ such that $\sup_{z \in \D_X} |x^* (z)| = \infty$, while $x^*$ is clearly uniformly continuous on $\D_X$. The converse is obvious.
\end{proof}

Note that $\ell_1 = c_0^*$ embeds isometrically in $\Ho^\infty (B_{c_0})$. As the Banach algebras $\Ho^\infty (B_{c_0})$ and $\Ho^\infty (\D_X)$ are isometrically isomorphic, the image of $\ell_1$ under the isometric isomorphism $\Psi$ from $\Ho^\infty (B_{c_0})$ and $\Ho^\infty (\D_X)$ forms a closed subspace of $\Ho^\infty (\D_X)$. As a matter of fact, we will see from the next result that the image $\Psi (\ell_1)$ consists of elements of $\Ho^\infty (\D_X)$ which are linear on $X$.

\begin{prop}\label{prop:image_of_ell_1}
Let $X$ be a Banach space with a normalized Schauder basis $(e_j)_{j=1}^{\infty}$. Then $\Psi (\ell_1) = \Ho^\infty (\D_X) \cap X^*.$
\end{prop}

\begin{proof}
Note that $\Psi$ maps linear functionals on $c_0$ to linear functionals on $X$. For if $y = (y_n)_{n=1}^\infty \in \ell_1$, then
\begin{equation*}
| \Psi (y) (x) | =\left| \sum_{n=1}^\infty y_n e_n^* (x) \right| \leq \left(\sup_{j\in\N} \|e_j^* \| \right) \sum_{n=1}^\infty |y_n| = \left(\sup_{j\in\N} \|e_j^* \| \right)\|y\|_{\ell_1},
\end{equation*}
for every $x \in B_X$. This implies that $\Psi (\ell_1) \subset \{ x^* \in \Ho^\infty (\D_X) : x^* \in X^* \}$.

Conversely, let $x^* \in \Ho^\infty (\D_X) \cap X^*$.
Choose $\theta_n \in \R$ such that $x^* (e_n) = |x^* (e_n)| e^{i\theta_n} $ for each $n \in \N$. Fix $\eps \in (0,1)$ and let $z_k = \sum_{j=1}^{k} (1-\eps) e^{-i\theta_j} e_j$ for each $k \in \N$. Then $z_k \in \D_X$ for each $k \in \N$ and
	\[
	\sum_{n=1}^{k} (1-\eps ) |x^* (e_n)| = |x^*(z_k)| \leq \|x^*\|_{\D_X}
	\]
	for every $k \in \N$ and $\eps \in (0,1)$. By letting $k \rightarrow \infty$ and $\eps \rightarrow 0$ we obtain that $\| (x^* (e_j) )_{j=1}^{\infty} \|_{\ell_1} \leq \|x^*\|_{\D_X}$.
	Thus, $(x^* (e_j) )_{j=1}^{\infty}$ belongs to $\ell_1$. On the other hand,
	\[
	|x^* (x) | = \left| \sum_{j=1}^\infty x^* (e_j) e_j^* (x) \right| \leq \sum_{j=1}^\infty | x^* (e_j)| = \| (x^* (e_j) )_{j=1}^{\infty} \|_{\ell_1}
	\]
	for every $x \in \D_X$; hence $\| (x^* (e_j) )_{j=1}^{\infty} \|_{\ell_1} = \|x^*\|_{\D_X}$ and $\Psi ((x^* (e_n) )_{n=1}^\infty ) = x^*$.
\end{proof}

	 Let us observe from $\big(\sup\limits_{j\in\N} \|e_j^* \| \big)^{-1} B_X \subset \D_X$ that we have
$\sup_{x \in B_X} |x^* (x)| \leq \left(\sup_{j\in\N} \|e_j^* \| \right) \| x^* \|_{\D_X}$ for every $x^* \in X^\sharp$. Moreover, if we consider the bounded linear operator $T : \Ho^\infty (\D_X) \rightarrow \Ho^\infty (B_X)$ defined as $(T f)(x) = f\big( \big(\sup\limits_{j\in\N} \|e_j^* \| \big)^{-1} x \big)$ for every $f \in \Ho^\infty (\D_X)$ and $x \in B_X$, then $\|T f \|_{B_X} \leq \|f\|_{\D_X}$ for every $f \in \Ho^\infty (\D_X)$ and $T$ is injective. However, $T$ is not a monomorphism (injection with closed image) in general. For example, let $X = \ell_p$ with $1 < p < \infty$. For each $n \in \N$, put $f_n = \Psi ( (\underbrace{n^{-1}, \ldots, n^{-1}}_{n\text{-many}}, 0,\ldots) ) \in \Ho^\infty (\ell_p \cap \D^\N)$. Then $\|f_n\|_{\ell_p\cap\D^\N} = \| ({n^{-1}, \ldots, n^{-1}}, 0,\ldots)\|_{\ell_1} = 1$ for every $n \in \N$. However,
\begin{align*}
|(T f_n) (y)| =\left| \sum_{j=1}^{n} \frac{1}{n} y_j \right| \leq \left( \sum_{j=1}^{n} \left(\frac{1}{n}\right)^q \right)^{\frac{1}{q}} \left( \sum_{j=1}^n |y_j|^p \right)^{\frac{1}{p}} \leq \frac{1}{n^{1-\frac{1}{q}}}
\end{align*}
for every $y \in B_{\ell_p}$, where $q$ is the conjugate of $p$. Thus, $\|T f_n \|_{B_{\ell_p}} \leq \frac{1}{n^{1-\frac{1}{q}}} \rightarrow 0$ as $n \rightarrow \infty$ while $\|f_n\|_{\ell_p \cap \D^\N} = 1$ for every $n \in \N$. This implies that $T : \Ho^\infty (\ell_p\cap\D^\N) \rightarrow \Ho^\infty (B_{\ell_p})$ is not a monomorphism.


\section{Fibers and cluster values of $\M (\Ho^\infty (\D_X))$}\label{sec:fiber}

In this section, we would like to study fibers of the spectrum $\M (\Ho^\infty (\D_X))$ which could help to describe the analytic structure of the spectrum. To this end, as we mentioned in Section \ref{Background}, we first consider the subspace $X^\sharp$ of the algebra $\Ho^\infty (\D_X)$ defined as follows as in \eqref{X_sharp_U}:
	\[
	X^\sharp = \left(\Ho^\infty (\D_X) \cap X^* , \|\cdot\|_{\D_X} \right).
	\]
	It is clear that $X^\sharp$ is a closed subspace of $\Ho^\infty (\D_X)$ endowed with the norm $\| \cdot \|_{\D_X}$ (thus, a Banach space).
	Moreover, the structure of the space $X^\sharp$ is not as rare as it might seem and Proposition \ref{prop:image_of_ell_1} proves that $X^\sharp$ is actually nothing but $\Psi (\ell_1)$; hence it is isometrically isomorphic to $\ell_1$. We have observed the following result.

\begin{prop}\label{ell_1:X^dagger}
	If $X$ is a Banach space with a normalized basis $(e_j)_{j=1}^{\infty}$, then
	\[
	\| x^* \|_{\D_X} = \sum_{j=1}^{\infty} |x^* (e_j)|
	\]
	for every $x^* \in X^\sharp$.
In particular, the mapping $\tau : X^\sharp \rightarrow \ell_1$ defined as $\tau (x^*) = (x^* (e_j))_{j=1}^{\infty}$ is an isometric isomorphism.
\end{prop}

To define fibers of $\M (\Ho^\infty (\D_X))$ as in Definition \ref{def:fibers_U}, we need to check that the restriction mapping $\pi : \M (\Ho^\infty (\D_X)) \rightarrow \overline{B}_{(X^\sharp)^*}$ is surjective in this case. Before proceeding further, let us clarify first some relationship between the mappings $\iota : X \rightarrow c_0$, $\kappa : c_{00} \rightarrow X$ (see the comment after the proof of Theorem \ref{Hol:D(X):B_{c_0}}) and the isometry $\tau: X^\sharp \rightarrow \ell_1$ in Proposition \ref{ell_1:X^dagger}.
Consider the following diagram.

\begin{remark}\label{relation_tau_iota'}
Let $X$ be a Banach space with a normalized Schauder basis.
\[
	\begin{tikzcd}
{c_{00}}\arrow{r}{\kappa} &X \arrow{r}{\iota} \arrow[d,-, dashed]{} & c_0 \arrow[d,-, dashed]{}& \, \, \\
\, & X^* & \ell_1 \arrow{l}{\iota^*} \arrow[d,-, dashed]{} & X^\sharp \arrow{l}{\tau} \arrow[d,-, dashed]{} & \Psi (\ell_1) \ar[equal]{l} \arrow[d,-, dashed]{}\\
\, & \, & \ell_\infty \arrow{r}{\tau^*} & (X^\sharp)^* & \Psi (\ell_1)^* \ar[equal]{l} \\
\end{tikzcd}
	\]	
We have that
(a) $\iota \circ \kappa = \text{Id}_{c_{00}}$; (b) $\tau = (\Psi \vert_{\ell_1})^{-1}$; (c) $\iota^* = \tau^{-1}$; (d) $\kappa = \tau^* \vert_{c_{00}}$.
\end{remark}

Note that an element $x$ in $\D_X$ can be naturally considered as an element in $(X^\sharp)^*$ which sends $x^* \in X^\sharp$ to $x^* (x)$ and we observe from (d) of Remark \ref{relation_tau_iota'} that the image of $B_{c_{00}}$ under the mapping $\tau^*$ is contained in $\D_X$. One might ask if $\tau^*$ sends elements in $B_{c_0}$ to elements in $\D_X$ as it does for elements in $B_{c_{00}}$. However, this is not the case when $X$ is $\ell_p$ with $1 \leq p < \infty$. Indeed, consider $u = \left( \frac{1}{2^{1/p}}, \frac{1}{3^{1/p}}, \ldots \right) \in B_{c_0}$. Then
\[
(\tau^* (u) )(e_n^*) = \langle u, \tau (e_n^*) \rangle = u_n = \frac{1}{(n+1)^{1/p}}
\]
for every $n \in \N$. It follows that if $\tau^* (u) \in \D_{\ell_p}$, then $\tau^* (u)$ would coincide with $u$. This  contradicts to that $u$ does not belong to $\ell_p$.

We now study the restriction mapping $\pi : \M (\Ho^\infty (\D_X) ) \rightarrow \overline{B}_{(X^\sharp)^*}$, i.e., $\pi (\phi) = \phi \vert_{X^\sharp}$ for every $\phi \in \M (\Ho^\infty (\D_X) )$.
The following remark proves that the restriction mapping $\pi: \M (\Ho^\infty (\D_X) ) \rightarrow \overline{B}_{(X^\sharp)^*}$ is surjective by presenting a relationship between the mapping $\pi $ and the restriction mapping $\pi_{c_0} : \M (\Ho^\infty (B_{c_0}) ) \rightarrow \overline{B}_{\ell_\infty}$.

\begin{remark}\label{relation_through_tau} Let $X$ be a Banach space with a normalized Schauder basis.
	Then the action between elements in the space $X^\sharp$ and the spectrum $\M (\Ho^\infty (\D_X))$ is related to the one between elements in $\ell_1$ and $\M (\Ho^\infty (B_{c_0}))$. Indeed,
	given $x^* \in X^\sharp$ and $\phi \in \Ho^\infty (\D_X)$, we have
\begin{align*}
\Phi (\phi) ( \tau (x^*))
&= \phi \left( x \in \D_X \leadsto \sum_{j=1}^{\infty} x^*(e_j) e_j^* (x) \right) = \phi (x^*).
\end{align*}
In other words, if we denote by $\pi_{c_0}$ the natural surjective mapping from $\M (\Ho^\infty ( B_{c_0}))$ to $\overline{B}_{\ell_\infty}$, then $	\pi (\phi) (\cdot) = \pi_{c_0} (\Phi (\phi)) (\tau (\cdot) ) \, \text{ on } X^\sharp.$
	Therefore, the following diagram commutes:
	\[
	\begin{tikzcd}[row sep=large]
\M (\Ho^\infty (\D_X) ) \arrow{r}{\Phi} \arrow[swap]{d}{\pi} & \M(\Ho^\infty ( B_{c_0})) \arrow{d}{\pi_{c_0}} \\
\overline{B}_{(X^\sharp)^*} & \overline{B}_{\ell_\infty} \arrow{l}{\tau^* \vert_{\overline{B}_{\ell_\infty}} }
\end{tikzcd}
	\]
	In particular, the restriction mapping $\pi$ is surjective. Moreover, if $(\phi_\alpha)$ is a net in $\M (\Ho^\infty (\D_X))$ that converges to some $\phi$ in $\M (\Ho^\infty (\D_X))$, i.e., $\phi_\alpha (f) \rightarrow \phi(f)$ for every $f \in \Ho^\infty (\D_X)$, then $\pi (\phi_\alpha) \rightarrow \pi (\phi)$ in the $w ( (X^\sharp)^*, X^\sharp )$-topology.
\end{remark}

As the space $X^\sharp$ is not trivial and the restriction mapping $\pi : \M (\Ho^\infty (\D_X)) \rightarrow \overline{B}_{(X^\sharp)^*}$ is surjective, we can restate Definition \ref{def:fibers_U} in the case of the Banach algebra $\Ho^\infty (\D_X)$ as follows.

\begin{definition}\label{def:fibers}
Let $X$ be a Banach space with a normalized basis. For $z \in \overline{B}_{(X^\sharp)^*}$, the \emph{fiber} $\M_z (\Ho^\infty (\D_X ))$ of the spectrum $\M (\Ho^\infty (\D_X ))$ at $z$ is defined as
\[
\M_z (\Ho^\infty (\D_X)) = \{ \phi \in \M (\Ho^\infty (\D_X)) : \pi (\phi) = z \} = \pi^{-1} (z).
\]
\end{definition}


On the other hand, we can view $\D_X$ as a subset of $\overline{B}_{(X^\sharp)^*}$ as each $x \in \D_X$ induces a linear functional with norm less than or equal to one on $X^\sharp$ which maps $x^* \in X^\sharp \mapsto x^* (x)$. From this point of view, we can consider the $w ((X^\sharp)^*, X^\sharp )$-closure of $\D_X$ in the dual space $(X^\sharp)^*$.
For convenience, we shall simply denote $w ((X^\sharp)^*, X^\sharp )$ by $\sigma$ throughout this paper.
In the following result, we show that the $\sigma$-closure of $\D_X$ coincides with the closed ball $\overline{B}_{(X^\sharp)^*}$; hence it is isometrically isomorphic to $\overline{B}_{\ell_\infty}$ as well.

\begin{prop}\label{pi:projection}
	Let $X$ be a Banach space with a normalized basis $(e_j)_{j=1}^{\infty}$. Then
	\begin{equation}\label{the_image_of_pi_projection}
	\overline{B}_{(X^\sharp)^*} = \{ z \in (X^\sharp)^* : |z(e_j^*)| \leq 1 \,\, \text{for every} \,\, j \in \N \} = \overline{\D}_X^{\sigma}.
	\end{equation}
\end{prop}

\begin{proof}
	The first equality in \eqref{the_image_of_pi_projection} is clear by definition of the isometry $\tau^*: \ell_\infty \rightarrow (X^\sharp)^*$.
	We claim that $\overline{B}_{(X^\sharp)^*} \subset \overline{\D}_X^{\sigma}$.
	Let $z \in \overline{B}_{(X^\sharp)^*}$ be given. Then there exists $u = (u_n) \in \overline{B}_{\ell_{\infty}}$ such that $\tau^* (u) = z$.
	Let $x^* \in X^\sharp$ be fixed. Then
	\[
	x^* \left( \sum_{j=1}^{m} \left(1-\frac{1}{m}\right) u_j e_j \right) = \sum_{j=1}^{m} \left(1-\frac{1}{m}\right) u_j x^* (e_j ) \rightarrow \sum_{j=1}^\infty u_j x^* (e_j) = z(x^*),
	\]
	as $m \rightarrow \infty$.
	This implies that the sequence $\big( \sum_{j=1}^{m} \left(1-\frac{1}{m}\right) u_j e_j \big)_{m=1}^\infty \subset \D_X$ converges in $\sigma$-topology to $z$.
	Thus, $\overline{B}_{(X^\sharp)^*} \subset \overline{\D}_X^{\sigma}$.
	
	Next, it is clear that $\D_X \subset \pi (\M (\Ho^\infty (\D_X)))$. As $\pi$ is $w^*$-$\sigma$-continuous and $\M(\Ho^\infty (\D_X))$ is compact, we get $	\overline{\D}_X^{\sigma} \subset \pi (\M (\Ho^\infty (\D_X)))$.
	Therefore, we obtain that $\overline{B}_{(X^\sharp)^*} \subset \overline{\D}_X^{\sigma} \subset \pi \left(\M (\Ho^\infty (\D_X) ) \right)= \overline{B}_{(X^\sharp)^*}$.
\end{proof}

The following result gives a description of the fibers of $\M (\Ho^\infty (\D_X))$ by revealing its relation with fibers defined from $\M (\Ho^\infty (B_{c_0}))$ in terms of the isometric isomorphism $\Phi$ between $\M (\Ho^\infty (\D_X) )$ and $\M (\Ho^\infty (B_{c_0}))$.

\begin{prop}\label{prop:fiber_representation}
Let $z \in \overline{\D}_X^{\sigma}$. Then we have
\begin{equation*}
\M_z (\Ho^\infty (\D_X)) = \Phi^{-1} \left( \M_{(\tau^{-1})^* (z)} (\Ho^\infty (B_{c_0})) \right) = \Phi^{-1} \left( \M_{(z(e_1^*), z(e_2^*), \ldots)} (\Ho^\infty (B_{c_0})) \right).
\end{equation*}
\end{prop}

\begin{proof}
Let $\phi \in \M_z (\Ho^\infty (\D_X))$. For $y \in \ell_1$, choose $x^* \in X^\sharp$ so that $\tau (x^*) = y$. Observe that
\begin{align*}
\Phi (\phi) (y) = \Phi (\phi) ( \tau (x^*)) &= \phi (x^*) \quad (\text{by Remark \ref{relation_through_tau}}) \\
&= (\tau^{-1})^* (z) (\tau (x^*)) = (\tau^{-1})^* (z) (y);
\end{align*}
thus, $\Phi (\phi) \in \M_{(\tau^{-1})^* (z)} (\Ho^\infty (B_{c_0}))$. 
\end{proof}

Let us remark that the size of fibers over points of $\overline{\D}_X^{\sigma}$ can be quite big as in the case for $\Ho^\infty (B_X)$. As a matter of fact, it is well known that the fiber $\M_u (\Ho^\infty (B_{c_0}))$ contains a copy of $\beta \N \setminus \N$ for every $u \in \overline{B}_{\ell_{\infty}}$ where $\beta \N$ is the Stone-\v{C}ech compactification of $\N$ \cite[Theorem 11.1]{ACG91}. So, one of direct consequences of Proposition \ref{prop:fiber_representation} is the following.

\begin{prop}
For any $z \in \overline{\D}_X^{\sigma}$, the set $\beta \N \setminus \N$ can be injected into the fiber $\M_z (\Ho^\infty (\D_X))$.
\end{prop}

Moreover, it is recently proved in \cite{CFGJM} and in \cite{DS} independently that for each $u \in \overline{B}_{\ell_\infty}$, there exists a Gleason isometric analytic injection of $B_{\ell_\infty}$ into the fiber $\M_u (\Ho^\infty (B_{c_0}))$ (see \cite{ADLM} for the definition of Gleason metric). Combining this result with Proposition \ref{prop:fiber_representation}, we can notice the following result.

\begin{prop}\label{cor:injection_into_the_fiber_2}
For any $z \in \overline{\D}_X^{\sigma}$, there exists a Gleason isometric analytic injection of $B_{\ell_{\infty}}$ into the fiber $\M_z (\Ho^\infty (\D_X))$.
\end{prop}


We finish this section by proving the Cluster Value Theorem for $\Ho^\infty (\D_X)$. To this end, we first need to define the cluster set for $f \in \Ho^\infty (\D_X)$ and $z \in \overline{\D}_X^{\sigma}$.

\begin{definition}
Let $X$ be a Banach space with a normalized basis. For each $z \in \overline{\D}_X^{\sigma}$ and $f \in \Ho^\infty (\D_X)$, the \emph{cluster set $Cl_{\D_X} (f, z)$ of $f$ at $z$} is defined as
\[
\{ \lambda \in \C : \, \text{there exists a net} \,\, (x_{\alpha}) \subset \D_X \,\, \text{such that} \,\, x_{\alpha} \xrightarrow{\,\sigma\,} z \,\, \text{and} \,\, f(x_\alpha) \rightarrow \lambda \}.
\]
\end{definition}

Notice that the following inclusion always holds:
\begin{equation}\label{CVT_inclusion_for_D_X}
Cl_{\D_X} (f,z) \subset \widehat{f} ( \M_{z} (\Ho^\infty (\D_X)) )
\end{equation}
for every $f \in \Ho^\infty (\D_X)$ and $z \in \overline{\D}_X^{\sigma}$. Indeed, if $\lambda \in Cl_{\D_X} (f, z)$, there exists a net $(x_\alpha) \subset \D_X$ such that $x_\alpha \xrightarrow{\,\sigma\,} z$ and $f(x_\alpha) \rightarrow \lambda$. Passing to a subnet, if necessary, we may assume that $\delta_{x_\alpha} \rightarrow \phi$ in $\M (\Ho^\infty (\D_X))$. Thus, $\widehat{f} (\phi) = \phi (f) = \lambda$. For $x^* \in X^\sharp$, note that $\phi (x^*) = \lim_\alpha \delta_{x_\alpha} (x^*) = z(x^*)$, which implies that $\pi (\phi) = z$.

\begin{definition}\label{def:CVT_D_X}
Let $X$ be a Banach space with a normalized basis. We say that the \emph{Cluster Value Theorem holds for $\Ho^\infty (\D_X)$ when both sets in \eqref{CVT_inclusion_for_D_X} coincide for every $f \in \Ho^\infty (\D_X)$ and $z \in \overline{\D}_X^{\sigma}$}.
\end{definition}

The following result, which points out the relation between cluster sets from $\Ho^\infty (B_{c_0})$ and the ones from $\Ho^\infty (\D_X)$, will be a key ingredient for our purpose.

\begin{prop}\label{prop:same_cluster_sets}
Let $X$ be a Banach space with a normalized basis. If $f \in \Ho^\infty (B_{c_0})$ and $u \in \overline{B}_{\ell_\infty}$, then $Cl_{B_{c_0}} (f, u) = Cl_{\D_X} (\Psi(f), \tau^*(u))$.
\end{prop}

\begin{proof}
Let $\lambda \in Cl_{B_{c_0}} (f, u)$ be given. By the separability of $\ell_1$, we may choose a sequence $(x_n) \subset B_{c_0}$ such that $x_n \xrightarrow{w(\ell_\infty, \,\,\ell_1)} u$ and $f(x_n) \rightarrow \lambda$. Taking $w_n \in B_{c_{00}}$ for each $n \in \N$ so that $\|x_n - w_n \| \rightarrow 0$ and $|f(x_n)-f(w_n)| \rightarrow 0$ as $n \rightarrow \infty$, we have that $w_n \xrightarrow{w(\ell_\infty, \,\,\ell_1)} u$ and $f(w_n) \rightarrow \lambda$. By Remark \ref{relation_tau_iota'} and the continuity in the $w(\ell_\infty,\ell_1)$-$\sigma$ topology of the mapping $\tau^*$,
\[
\kappa (w_n) = \tau^* (w_n) \xrightarrow{\,\sigma\,} \tau^* (u) \,\text{ and }\, \Psi (f) (\kappa (w_n) ) = f ( \iota (\kappa (w_n))) = f(w_n) \rightarrow \lambda.
\]
It follows that $\lambda \in Cl_{\D_X} (\Psi(f), \tau^*(u))$; thus $Cl_{B_{c_0}} (f, u) \subset Cl_{\D_X} (\Psi(f), \tau^*(u))$.

Conversely, let $\lambda \in Cl_{\D_X} (\Psi(f), \tau^*(u))$. By \eqref{CVT_inclusion_for_D_X}, there exists $\phi \in \M_{\tau^* (u)} (\Ho^\infty (\D_X))$ such that $\lambda = \phi (\Psi (f))$. Note that $\lambda = \Phi (\phi ) (f)$ and from Proposition \ref{prop:fiber_representation} that
\[
\Phi (\phi) \in \Phi \left( \M_{\tau^* (u)} (\Ho^\infty (\D_X)) \right) = \M_u (\Ho^\infty ( B_{c_0})).
\]
This implies that $\lambda \in \widehat{f} ( \M_u (\Ho^\infty ( B_{c_0})) )$. As the Cluster Value Theorem holds for $\Ho^\infty (B_{c_0})$ \cite[Theorem 5.1]{ACGLM}, we obtain that $\lambda \in Cl_{B_{c_0}} (f, u)$. Hence, we conclude that $Cl_{\D_X} (\Psi(f), \tau^*(u)) \subset Cl_{B_{c_0}} (f, u)$.
\end{proof}

Note that Proposition \ref{prop:same_cluster_sets} and \cite[Lemma 2.1]{ACGLM} imply that each cluster set $Cl_{\D_X} (f, z)$ is a compact connected set. Now, we are ready to prove the Cluster Value Theorem for $\Ho^\infty (\D_X)$ provided that a Banach space $X$ has a normalized basis.

\begin{theorem}\label{thm:CVT_}
Let $X$ be a Banach space with a normalized basis. Then Cluster Value Theorem holds for $\Ho^\infty (\D_X)$.
\end{theorem}

\begin{proof} Let $z \in \overline{\D}_X^{\sigma}$, $f \in \Ho^\infty (\D_X)$ and $\phi \in \M_z (\Ho^\infty (\D_X))$ be given.
Thanks to the inclusion \eqref{CVT_inclusion_for_D_X}, it suffices to prove $\lambda := \phi (f) \in Cl_{\D_X} (f,z)$. Note from Proposition \ref{prop:fiber_representation} that $\Phi (\phi) \in \M_{(\tau^{-1})^* (z)} (\Ho^\infty (B_{c_0}))$. From the Cluster Value Theorem for $\Ho^\infty (B_{c_0})$ at every $u \in \overline{B}_{\ell_{\infty}}$,
we get that
\begin{align*}
\lambda = \phi(f) = \Phi (\phi) ( \Psi^{-1} (f) ) \in Cl_{B_{c_0}} ( \Psi^{-1} (f), (\tau^{-1})^* (z)).
\end{align*}
By Proposition \ref{prop:same_cluster_sets}, we have that
\[
Cl_{B_{c_0}} ( \Psi^{-1} (f), (\tau^{-1})^* (z)) = Cl_{\D_X} (f, z).
\]
Thus, $\lambda \in Cl_{\D_X} (f,z)$ and it completes the proof.
\end{proof}

		\noindent \textbf{Acknowledgment:\ } We would like to thank Daniel Carando and Veronica Dimant for valuable comments and fruitful conversations. We are also grateful to anonymous referees for several helpful questions and suggestions.

	\end{document}